\documentclass[letterpaper, 10 pt, conference]{ieeeconf}  
\IEEEoverridecommandlockouts                            
\usepackage{hyperref}
\usepackage{amsmath} 
\usepackage{amssymb}  
\usepackage{xcolor}
\usepackage{graphicx}
\usepackage{subcaption}
\newtheorem{theorem}{Theorem}[section]

\newtheorem{lemma}[theorem]{Lemma}

\newcommand{\mat}[1]{\ensuremath{\begin{bmatrix}#1\end{bmatrix}}}

\DeclareMathOperator*{\ext}{ext}
\usepackage{xfrac}
\usepackage{colortbl}	

\usepackage[linesnumbered,ruled]{algorithm2e}

\SetCommentSty{mycommfont}

\SetKwInput{KwInput}{Input}                
\SetKwInput{KwOutput}{Output}              

\usepackage{multirow}

\usepackage[normalem]{ulem}
\useunder{\uline}{\ul}{}

\title{\LARGE \bf
iRiSC: Iterative Risk Sensitive Control for Nonlinear Systems with Imperfect Observations
}

\author{Bilal Hammoud$^{1,2}$, Armand Jordana$^{1}$ and Ludovic Righetti$^{1,2}$

\thanks{*This work was supported by New York University, the European Union’s Horizon 2020 research and innovation program (grant agreement 780684) and the National Science Foundation (grants 1825993, 1932187, 1925079 and 2026479).}
\thanks{$^{1}$Tandon School of Engineering,
        New York University, USA {\tt\small bah436,aj2988,lr114@nyu.edu}}%
\thanks{$^{2}$Max Planck Institute for Intelligent Systems, Tübingen, Germany}
}

\begin{document}

\maketitle
\thispagestyle{empty}
\pagestyle{empty}

\begin{abstract}
This work addresses the problem of risk-sensitive control for nonlinear systems with imperfect state observations, extending results for the linear case.
In particular, we derive an algorithm that can compute local solutions with computational complexity similar to the iterative linear quadratic regulator algorithm. The proposed algorithm introduces feasibility gaps to allow the initialization with non-feasible trajectories. Moreover, an approximation for the expectation of the general nonlinear cost is proposed to enable an iterative line search solution to the planning problem. The optimal estimator is also derived along with the controls minimizing the general stochastic nonlinear cost. Finally extensive simulations are carried out to show the increased robustness the proposed framework provides when compared to the risk neutral iLQG counter part. To the authors' best knowledge, this is the first algorithm that computes risk aware optimal controls that are a function of both the process noise and measurement uncertainty. 
\end{abstract}


\section{Introduction}

Recently there has been an increased interest in computing optimal decisions that reason not only about the mean of a certain outcome or cost~\cite{sideris2005efficient, todorov2005generalized} but also about the higher order statistics of the problem. A particular class of such algorithms is known as risk sensitive optimal control, which takes into account the variability of a certain cost and is concerned with the infrequent occurrences of undesired events besides the frequent most common outcomes. Risk sensitive problems have been widely used in economics~\cite{arrow1971theory,pratt1978risk} and have recently gained popularity in robotics~\cite{hammoud2021impedance, majumdar2020should} and reinforcement learning \cite{Bechtle_Lin_Rai_Righetti_Meier_2019}. 

In linear quadratic Gaussian (LQG) control~\cite{bryson2018applied}, the expectation of a quadratic cost function $\mathbb{E}\left[\mathcal{L}\right]$ is minimized over the decision variables. The main result in LQG theory, is that the optimal control is the same as that of the deterministic linear quadratic regulator (LQR) case. This is not the case anymore in the linear risk-sensitive control case with noisy measurements.
%
%
%

The risk-sensitive optimal control problem of interest in this paper was introduced by Jacobson~\cite{jacobson1973optimal} which instead minimizes the cost
\begin{align}
    \mathcal{J} = - \sigma^{-1} \ln{\mathbb{E}\left[e^{-\sigma \mathcal{L}} \right]} 
\end{align}
where $\sigma$ is a scalar whose role will become clear later. Under this particular transformation, and in the limit as $\sigma \to 0$ this transformed cost can be approximated as 
\begin{align}
    \mathcal{J} \approx \mathbb{E}\left[\mathcal{L}\right] - \frac{\sigma}{2} Var\left[\mathcal{L}\right] + \ldots 
\end{align}
where $\ldots$ contains terms that are functions of higher order moments of $\mathcal{L}$. For $\sigma < 0$ (risk-averse case) the variability of the original cost $\mathcal{L}$ contributes positively to the newly transformed cost, thus an optimizer will seek a solution reducing the variability of the original cost and minimizing the possibility of infrequent events occurring. The opposite is true for the case of $\sigma > 0$ (risk-seeking case). Other forms of risk sensitivity have been studied~\cite{singh2018framework, tsiamis2020risk} and will not be considered in this paper. The solution for the stochastic optimal control problem under the exponential transformation given linear process model with additive Gaussian noise and a quadratic cost which we shall denote LQEG was derived by Jacobson~\cite{jacobson1973optimal}.

The case with noisy measurements (and no direct access to the state) was studied in Speyer~\cite{speyer1974optimization}, who introduced a solution that grows with the history of observations for linear process and measurement models. The problem involving linear process and measurement models was eventually solved satisfactorily by Whittle~\cite{whittle1981risk} for the quadratic cost case. Variations of this problem were later studied either for the case of partially observable systems~\cite{bensoussan1985optimal, james1994risk} or the nonlinear case with full state knowledge~\cite{roulet2020convergence}. Ponton~\cite{ponton2020effects} proposed an iterative algorithm to solve the case with measurement noise that requires augmenting the true state dynamics with the dynamics of an Extended Kalman Filter. However the derivation was heuristic with an incorrect line-search step which led to difficult to converge iterations and an incorrect update of the feedforward control.


This work aims to provide the first iterative algorithm
that computes a locally optimal solution to the nonlinear risk-sensitive optimal control problem with measurement uncertainty.
We extend the results of Whittle~\cite{whittle1981risk} to the general nonlinear case of process and measurement dynamics with additive Gaussian noise.
We propose a two-stage algorithm, where the first stage computes a nominal trajectory given the observations available at the starting point using an efficient line-search method. The second stage computes the optimal estimator along with the control optimizing the risk sensitive cost function as measurements are made during execution.
Finally, extensive simulations of a nonlinear dynamical system involving stiff interactions with the environment demonstrate the capability of such controllers.

\textbf{Notation:} The gradient of a function $f$ with respect to a vector $v$ is denoted as $f^v$, similarly for second order derivatives w.r.t vectors $u,v$ will be denoted as $f^{uv}$. The determinant of a matrix $M$ will be denoted by $\vert M \vert$. All functions are assumed to be $\mathcal{C}^2$ unless otherwise stated. If $(v_i)_{i\in \mathbb{N}}$ is a sequence of vectors, then $v_{k:t}$ denotes the batch vector of all $v_i$ for $k \leq i \leq t$. 
\section{background}\label{sec:background}

This section gives an overview of the Linear Quadratic Exponential Gaussian problem (LQEG), where both process and measurement models are linear and the cost is an exponential of a quadratic in the state and control variables. We go over the main results from Jacobson~\cite{jacobson1973optimal} and Whittle \cite{whittle1981risk} which we will build upon for our extension to the general nonlinear problem. Consider the following discrete time process and observation models, 
\begin{subequations}
\begin{align}
    x_{t+1} &= f^x_t x_t + f^u_t u_t + \omega_{t+1} \\
    y_{t+1} &= h^x_t x_t + \gamma_{t+1}
\end{align}
\end{subequations}
where $x_t$, $y_t$ and $u_t$ are the state, observation and control at time $t$ respectively. Given a belief $\hat x_0$, the initial state is assumed to follow a Gaussian distribution  $x_0 \sim \mathcal{N}\left(\hat{x}_0, \chi_0 \right)$. The process and measurement disturbances are considered to be Gaussian and denoted by $\omega_{t} \sim \mathcal{N}\left(0, \Omega_t\right)$ and $\gamma_{t} \sim \mathcal{N}\left(0, \Gamma_t\right)$.
Then, given control inputs (or a policy), the joint probability density of the entire trajectory can be written
\begin{align}
     p\left(x_0, w_{1:T}, \gamma_{1:T}\right) &= \frac{1}{\kappa} \exp\left(- \mathcal{D} \right) \\
    \mbox{where} \ \ \ \mathcal{D} &= d_0 + \sum_{t=1}^{T} d_t \nonumber \\
     d_0 &= \frac{1}{2}(x_0 -\hat{x}_0)^T \chi^{-1}_{0} (x_0 -\hat{x}_0) \nonumber \\ 
    d_t &= \frac{1}{2}\omega_t^T \Omega_t^{-1} \omega_t + \frac{1}{2} \gamma_t^T \Gamma_t^{-1} \gamma_t \nonumber \\
     \kappa &= \vert 2 \pi \chi_{0} \vert ^{\frac{1}{2}}  \prod_{t=1}^T \vert 2 \pi \Omega_t \vert ^{\frac{1}{2}} \vert 2 \pi \Gamma_t \vert ^{\frac{1}{2}} \nonumber 
\end{align}

 Observations available at time $t$ are denoted by $W_t$ and take the form 
\begin{equation}
W_t = \left(\hat{x}_0, y_{1:t}, u_{0:t-1}\right). \label{eqn:observation-vector}
\end{equation}
It is important to note that the observation history also includes $ \hat{x}_0 $ the belief of the initial state.
Then, the risk-sensitive optimal control problem of the imperfectly observable system at time~$t$ becomes a problem of computing the optimal policy $\pi$ to minimize the following cost function
\begin{align}
    \mathcal{J}(\pi, W_t) &= - \sigma^{-1} \ln \left( \mathbb{E}_\pi \left[ \exp\left(-\sigma \mathcal{L}\right) \ \vert \ W_t   \right] \right) \label{eqn:cost_definition}  \\ 
    \mbox{where} \ \mathcal{L}  &= \frac{1}{2}  x_T^T l^{xx}_T x_T + \frac{1}{2} \sum_{t=0}^{T-1} x_t^T l^{xx}_t x_t + u_t^T l^{uu}_t u_t \nonumber \label{eqn:functional} 
\end{align}
where the expectation is taken over every $w_t$, $\gamma_t$ and $x_0$. Taking the expectation conditioned on $W_t$ determines $ u_{0:t-1}$, however, control inputs, $ u_{t:T}$, are taken according to the policy $\pi$. We consider policies that only depend on previous observations. Before time~$t$, observations are determined by $W_t$ but after time~$t$, the observations are random variable therefore future control inputs are also random variables. The optimal policy is then a function of the current available observations
\begin{equation}
    \pi^{\star}(W_t) = \underset{\pi}{\operatorname{argmin}} \,\, \mathcal{J}(\pi, W_t) \label{eqn:pi_star_definition}
\end{equation}

\subsection{Optimal Control with Imperfect Observations}

The expectation at time $t=0$ of the total cost of the control problem can be expressed as 
\begin{align}
    \mathbb{E}_{\pi}\left[e^{-\sigma \mathcal{L}}\right \vert W_0] &= \int e^{-\sigma \mathcal{L}} p\left(x_0, w_{1:T}, \gamma_{1:T}\right) dx_0 dw_{1:T}  d\gamma_{1:T} \nonumber  \\ 
    &= \frac{1}{\kappa} \int e^{-\sigma \left(\mathcal{L} + \sigma^{-1} \mathcal{D}\right)}  dx_0 dw_{1:T}  d\gamma_{1:T} \label{eqn:cost-expectation-definition}
\end{align}
Defining the total stress as $\mathcal{S} = \mathcal{L} + \sigma^{-1} \mathcal{D}$. 
Whittle then uses the quadratic lemma introduced by Jacobson~\cite{jacobson1973optimal} and summarized in Appendix~\ref{lemma:integration_eofq} in order to compute the expectation in~\eqref{eqn:cost-expectation-definition}. The main idea is that an integration over an exponential of a quadratic of a variable $x$ can be replaced by an optimization over the exponential of a quadratic of this variable. Then using induction, Whittle \cite{whittle1981risk} proves that  

\begin{align}
    \mathbb{E}_{\pi^{\star}} \left[ e^{\left(-\sigma \mathcal{L}\right)} \ \vert \ W_t   \right] &= \frac{\alpha_t}{p(W_t)} e^{\left( - \sigma \mathcal{S}_t(W_t) \right)}
\end{align}
where $\alpha_t$ is provided in appendix~\ref{app:principle_of_optimality} and
\begin{subequations}
\begin{align}
    \mathcal{S}_t(W_t) &= \min_{u_t} \ext_{y_{t+1}} \mathcal{S}_{t+1}(W_{t+1}) \\
    \mathcal{S}_T(W_T) &= \ext_{x_0, ..., x_T} \mathcal S \label{eqn:req_S}
\end{align}    
\end{subequations}
Here $\ext$ denotes an extremization, i.e. computation of either a maximum or a minimum. 
We say that a variable extremizes $\mathcal{S}$ if it minimises $\mathcal{S}$ when $\sigma < 0$ and  maximizes $\mathcal{S}$ when $\sigma > 0$.
This is due to the fact that the total stress could possess either a maximum or a minimum in the random variables depending on the sign of $\sigma$. 
Obtaining the optimal policy then boils down to the following minimization
\begin{align}
  \min_{u_t, ..., u_{T-1}} \ext_{x_0, ..., x_T}\ext_{y_{t+1}, ..., y_T} \mathcal S  
\end{align}    

The optimal $u_t$ is then the optimal control at time~$t$ given the current observations. Whittle's certainty equivalence principle says that the order of the successive extremization and minimization can be interchanged. If $\sigma$ is negative, then the minimization problem is well defined only if $\mathcal S $ is negative definite in the extremizing variables. This condition ensures that the optimal cost is finite. More precisely, there exist a threshold value $\Bar{\sigma} < 0 $ depending on the problem parameters such that for all $\sigma > \Bar{\sigma}$, the problem \ref{eqn:pi_star_definition} is well defined.
The implications of this condition on $\sigma$ will become obvious later.

We provide a brief proof of these results in Appendix~\ref{app:principle_of_optimality}. Unlike the iterative linear quadratic Gaussian iLQG~\cite{todorov2005generalized}, or the fully observable case of LQEG~\cite{jacobson1973optimal}, the optimal control cannot be obtained with a single backward recursion in a dynamic programming fashion. This is due to the fact that the terminal condition~\eqref{eqn:req_S} includes an optimization over the entire state trajectory. 

\subsection{Separation principle}

As Whittle concluded, at time $t$ the problem includes minimizing over future controls, and extremizing over future observations and the entire state trajectory. Whittle~\cite{whittle1981risk} then suggested breaking the problem into two sub problems while keeping the current state $x_t$ as a free variable, thus introducing the past stress $\mathcal{P}\left(x_t, W_t\right)$ and future stress $\mathcal{F}\left(x_t\right)$ at each time $t$\footnote{Note that $\mathcal{P}$ and $\mathcal{F}$ should be index by $t$, but in order to simplify the notations, we omit this time index as the latter can be inferred thanks to~$x_t$.}, which are defined by  
\begin{align}
    \mathcal{P}\left(x_t, W_t\right) &= \underset{x_{0:t-1}}{\ext} \  \sum^{t-1}_{i=0} l_i + \sigma^{-1} \sum^{t}_{i=0} d_i, \label{past_stress_reccursion} \\ 
    \mathcal{F}\left(x_t\right) &= \min_{u_{t:T-1}} \ \underset{y_{t+1:T}}{\underset{x_{t+1:T}}{\ext}} \ \ \sum^{T}_{i=t} l_i + \sigma^{-1} \sum^{T}_{i=t+1} d_{i} .
\end{align}

The past and future stress problems can then be computed sequentially separately as 
\begin{align}
    \mathcal{P}\left(x_t, W_t\right) &= \underset{x_{t-1}}{\ext} \ \ l_{t-1} + \sigma^{-1} d_{t} +  \mathcal{P}\left(x_{t-1}, W_{t-1}\right), \\
    \mathcal{F}\left(x_t\right) &= \min_{u_t} \underset{x_{t+1}}{\ext} \ \  l_{t} + \sigma^{-1} d_{t+1} +  \mathcal{F}\left(x_{t+1}\right) ,
\end{align}
where Whittle shows that optimizing over future observations $y_{t+1:T}$ yields a prediction of the form $y_{t+1} = h^x_t x_t$. In other words, for all $ h > t+1$, $\gamma_t = 0$. The recursions have the following boundary conditions:

\begin{align}
  \sigma   \mathcal{P}(x_0, W_0) &= \frac{1}{2}(x_0 -\hat{x}_0)^T \chi^{-1}_0 (x_0 -\hat{x}_0), \\
    \mathcal{F}(x_T) &=  \frac{1}{2} x_T^T l^{xx}_T x_T. 
\end{align}

As the the current state $x_t$ has been kept as a free variable in optimization, the optimal control can be written as a function of $x_t$: $u_t = \pi_{x_t} (W_t)$. $x_t$ is then determined by the last extremization
\begin{align}
    \check{x}_t = \underset{x_t}{\operatorname{argmin}} \ \ \mathcal{P}\left(x_t, W_t\right)  + \mathcal{F}\left(x_t\right)
\end{align}
The solution $\check{x}_t$ of this extremization then gives us the optimal control: 
$u_t^{\star} = \pi_{\check{x}_t} (W_t) = \pi^{\star} (W_t)$. 
This means that the optimal control is not obtained by only solving a backward recursion, but instead by optimizing an entire trajectory of states, while the policy as a function of these optimal states is obtained during the backward recursion. We will appeal next to these results to deriving an iterative algorithm for the nonlinear case.

\section{Algorithm Overview}\label{sec:overview}
To solve the nonlinear problem, we propose an algorithm consisting of two stages. The first stage (discussed in details in Sec.~\ref{sec:search_algorithm}) optimizes a nominal trajectory $\mathcal{X}, \mathcal{U}$ and a policy $\pi(\delta \check{x}_t)$ where $\delta \check{x}_t$ is the deviation of the state minimizing the total stress $\mathcal{S}$ from the nominal trajectory. The second stage (discussed in Sec.~\ref{sec:estimation}) is executed at run-time as measurements are made available and includes computing the minimum stress deviations $\delta \check{x}_t$ along with the corresponding optimal policy.  

Then the algorithm proceeds at follows, first, before run-time, the nominal trajectory and the policy $\pi$ are optimized according to Algorithm~\ref{alg:planning} until convergence. Once Algorithm~\ref{alg:planning} has converged, the reference trajectory $\mathcal{X}$ along with the policy parameters $k_t$ and $K_t$ and the future stress Hessian and gradient along the trajectory $V_t$ and $v_t$ are stored for the second stage. 

The second stage happens at run-time. As each observations $y_t$ becomes available, the past stress is propagated one time step forward as described in Algorithm~\ref{alg:estimation} to obtain $\hat{x}_t$ and $P_t$, then the minimum stress estimate $\check{x}_t$ at time $t$ is computed, and the optimal control is obtained and executed.   

\section{Nominal Trajectory Optimization}\label{sec:search_algorithm}

Consider the general nonlinear process and measurement models with independent identically distributed (i.i.d) additive Gaussian noise $\omega_t$ and $\gamma_t$

\begin{subequations}
\begin{align}
    x_{t+1} &= f_t \left(x_t, u_t\right) + \omega_{t+1}  \\ 
    y_{t+1} &= h_t\left(x_t\right) + \gamma_{t+1}
\end{align}
\end{subequations}
where $f_t$ and $h_t$ are differentiable. Along with the general nonlinear cost function of the form 
\begin{align}
    \mathcal{L} = l_T\left( x_T\right) + \sum_{0}^{T-1} l_t\left( x_t, u_t\right) \label{eqn:functional} 
\end{align}
where $l_t\left( x_t, u_t\right)$ is any twice differentiable nonlinear function of the states and controls.
With this cost, the minimization \eqref{eqn:pi_star_definition} generally leads to an intractable optimal control problem that might possess multiple local minima~\cite{chapman2021classical}. A common approach to find one such local minimum is to follow an iterative approach similar to that of iterative linear quadratic regulator $\cite{sideris2005efficient}$ and perform a line search on the cost function~\cite{nocedal2006numerical} until some convergence criteria is achieved. So, we seek some iterative updates to the state and control trajectories in the form 

\begin{align}
    \mat{\mathcal{X} \\ \mathcal{U}}^{i+1} = \mat{\mathcal{X} \\ \mathcal{U}}^{i} + \alpha \mat{ \delta \mathcal{X} \\  \delta \mathcal{U}}^{i} \label{eqn:nominal_update}
\end{align}

$\mathcal{X}$ and $\mathcal{U}$ denote the nominal state and control trajectories at the indicated iterations respectively. The $\delta$ is used to denote a small change around these trajectories and $\alpha$ is the step length corresponding to the search direction $\delta \mathcal{X}, \delta \mathcal{U}$ and determined via line search.

\subsection{Step Direction}

In order to generate a descent direction for the line search a linear approximation of the dynamics and observations along the nominal trajectory is computed and a quadratic approximation of the cost is considered.  These approximations are computed along the nominal trajectory at iteration $i$ denoted by $\mathcal{X}^i , \mathcal{U}^i$ and are given by  
\begin{subequations}
\begin{align}
    x_{t+1}  &\approx f_t(x_t^{i}, u_t^{i}) + f^x_t \delta x_t + f^u_t \delta u_t + \omega_{t+1}\\
    y_{t+1} &\approx h_t( x_t^{i}) + h^x_t \delta x_t +  \gamma_{t+1}
\end{align}
\end{subequations}
and
\begin{align}
    l_t\left(x_t^{i} + \delta x_t, u_t^{i} + \delta  u_t\right) \approx l_t\left(x_t^{i} , u_t^{i} \right) + \delta l_t
\end{align}
where $ \delta x_t = x_t - x_t^{i}$,  $ \delta u_t = u_t - u_t^{i}$ and $ \delta y_t = y_t - h_t( x_t^{i})$. In the remainder of the paper we will omit the dependence of $f(x,u)$ and $l(x,u)$ on $x$ and $u$ for brevity. The quadratic approximation of the cost is then given by  
\begin{align}
   \delta  l_t &= \frac{1}{2} \mat{1 \\  \delta x_t \\ \delta  u_t}^T \mat{0 & l^{x^T}_t & l^{u^T}_t \\ l^{x}_t & l^{xx}_t & l^{xu}_t \\ l^{u}_t & l^{ux}_t & l^{uu}_t  } \mat{1 \\ \delta x_t \\ \delta u_t} \label{eqn:cost_approx}
\end{align}

Following~\cite{giftthaler2018family} we introduce feasibility gaps along the nominal trajectory to form 
\begin{align}
    \Bar{f}_{t+1} = f(x_t^{i}, u_t^{i}) - x_{t+1}^{i}
\end{align}

Hence, we optimize $\mathbb{E}_\pi \left[ e^{-\sigma \delta \mathcal{L}}\right] $ subject to the process and observation deviation dynamics 
\begin{subequations}\label{eqn:dynamics_approx}
\begin{align}
   \delta  x_{t+1} &= f^x_t  \delta x_t + f^u_t \delta  u_t + \Bar{f}_{t+1}  +  \omega_{t+1} \label{eq:discrete_linearized_dyamics}\\ 
    \delta y_{t+1} &= h^x_t  \delta x_t +  \gamma_{t+1}
\end{align}
\end{subequations}

The optimal search direction at iteration $i$ is obtained by computing the optimal control deviation $\delta u^\star_t$ according to the future stress. We provide an alternative derivation to that of Whittle's in an extended document provided along with the code in~\footnote{\url{https://github.com/hammoudbilal/irisc}}. The document proves that the future stress optimization is independent of the future observations. The optimal control deviations are then given by the following theorem
\begin{theorem}\label{thm:optimal-controls}
Assuming that $\mathcal{F}(\delta x_{t+1})$ has the form 

\begin{align}
    \mathcal{F}(\delta x_{t+1}) = \frac{1}{2} \delta x^T_{t+1}V_{t+1} \delta x_{t+1} + {\delta x}^T_{t+1} v_{t+1} +  \Bar{v}_{t+1} 
\end{align}

Then the optimal control deviations $\delta u^\star_t$ minimizing $\mathbb{E}_\pi \left[ e^{-\sigma \delta \mathcal{L}} \vert W_0\right] $ are given recursively by  

\begin{align}
    \delta u^{\star}_t &= - \underbrace{Q^{uu^{-1}}_t Q^{u}_t}_{k_t} - \underbrace{Q^{uu^{-1}}_t Q^{ux}}_{K_t} \delta x_t  \label{eqn:optimal-control-deviation}
\end{align}

where $Q^u_t$, $Q^{uu}_t$ and $Q^{ux}_t$ are defined in~\eqref{eqn:Q_function}. \\ 
\end{theorem}

\begin{proof}

We start by writing the future stress, we know from \cite{whittle1981risk} that the measurement uncertainty yields a prediction and that the future stress takes the form 
\begin{align}
    \mathcal{F}\left( \delta x_t\right) = &\min_{\delta u_t}   \  \underset{\delta x_{t+1}}{\ext} \  l_t + \mathcal{F}(\delta x_{t+1}) \\ 
    &+ \frac{\sigma^{-1}}{2}  \left(\delta x_{t+1} -z_t\right)^T \Omega_{t+1}^{-1}\left(\delta x_{t+1} - z_t\right) \nonumber 
\end{align}
where $z_t = f^x_t  \delta x_t + f^u_t \delta  u_t + \Bar{f}_{t+1}$. The first step is to optimize with respect to $\delta x_{t+1}$ and rearrange the partially optimized future stress into a quadratic in the current controls and states $\delta u_t$ and $\delta x_t$ respectively

\begin{align}
\hspace{-0.05cm} \mathcal{F}\left(\delta x_t\right) &=  \min_{\delta u_t}  \frac{1}{2} \mat{1 \\ \delta  x_t \\  \delta u_t}^T \mat{ \Bar{q}_t & Q^{x^T}_t & Q^{u^T}_t \\ Q^{x}_t & Q^{xx}_t & Q^{xu}_t \\ Q^{u}_t & Q^{ux}_t & Q^{uu}_t  } \mat{1 \\  \delta x_t \\  \delta u_t} \label{eqn:partial-opt-future-sress}
\end{align}

 where 
 \begin{subequations}\label{eqn:Q_function}
 \begin{align}
     Q^{x}_t &= l^{x}_t + f^{x^T}_t M_t \Bar{f}_{t+1} + f^{x^T}_t N_t   \\ 
     Q^{u}_t &= l^{u}_t + f^{u^T}_t N_t + f^{u^T}_t M_t \Bar{f}_{t+1} \\ 
     Q^{uu}_t &= l^{uu}_t + f^{u^T}_t M_t f^{u}_t  \\ 
     Q^{ux}_t &=l^{ux}_t + f^{u^T}_t M_t f^{x}_t \\ 
     Q^{xx}_t &= l^{xx}_t + f^{x^T}_t M_t f^{x}_t \\ 
     M_t &= \left(\sigma \Omega_{t+1} + V^{-1}_{t+1}\right)^{-1}  \label{eqn:Mt} \\ 
    N_t &=  v_{t+1} - \sigma M_t  \Omega_{t+1} v_{t+1} \label{eqn:Nt}
 \end{align}
  \end{subequations}
  
And where $\Bar{q}_t$ is a constant. Being quadratic in $\delta u_t$, the remaining optimization is obtained by solving $\nabla_{\delta u_t} \mathcal{F}(\delta x_t) = 0$ which concludes the proof for the control recursions. A more detailed version of the proof that includes straightforward yet lengthy algebraic manipulations is provided in our public repository along with the code implementing all the simulations. 
\end{proof}

\begin{theorem}
The future stress is a quadratic in the state deviations $\delta x_t$ and takes the form 

\begin{align}
    \mathcal{F}(\delta x_t) &= \frac{1}{2} \delta x^T_t V_{t} \delta x_t + \delta x^T_t v_t +  \Bar{v}_t 
\end{align}

where 
\begin{subequations}\label{eqn:value_recursion}
\begin{align}
    V_t &= Q^{xx}_t + K^T_t Q^{uu}_t K_t + K^T_t Q^{ux}_t + Q^{xu}_t K_t \\ 
    v_t &= Q^{x}_t + K^T_t Q^{uu}_t k_t + K^T_t Q^{u}_t + Q^{xu}_t k_t \\
    \Bar{v}_t  &= \frac{1}{2} \Bar{q}_t + k^T_t Q^{u}_t + \frac{1}{2}k^T_t Q^{uu}_t k_t    
\end{align}
\end{subequations}
\end{theorem}
\vspace{0.2cm}

\begin{proof}
The proof follows directly by replacing the optimal controls in $\mathcal{F}(\delta x_t)$ and regrouping, details of the calculations are provided in the accompanying repository. 
\end{proof}


\subsection{Approximating the Cost Expectation}

In order to implement a line search procedure, it is necessary to evaluate the nonlinear risk sensitive cost. However, for a general nonlinear cost, computing the expectation is generally intractable analytically. Therefore, we use a quadratic approximation of the cost and a linearization of the dynamics to obtain an approximation of the risk sensitive cost. The linearization and quadratic approximations are done only over the states as the candidate control trajectory is considered to be fixed. More precisely, given a candidate control sequence $\mathcal{U}$ and a nominal trajectory $\mathcal{X}$, we evaluate:
\begin{equation}
\Tilde{\mathcal{J}} = - \sigma^{-1} \ln \left( \mathbb{E}  \left[ e^{-\sigma   \Tilde{\mathcal{L}}}  \vert u_{0:T-1} \right]  \right)
\end{equation} subject to 
\begin{subequations}
\begin{align}
    \delta x_0 &= x_0 - \hat x_0\\
    \delta x_{t+1} &= f^x_t \delta x_t + \Bar{f}_{t+1} +  \omega_{t+1} \label{eq:discrete_linearized_dyamics_for_cost_evaluation}
\end{align}
\end{subequations}
with: 
\begin{align}
   \Tilde{\mathcal{L}}  &= \sum_{t=0}^T \ \underbrace{l(x_t^n, u_t)}_{\bar{l}_t}  +  \underbrace{\frac{1}{2}  \delta x_t^T  l^{xx}_t  \delta x_t + l^{x^T}_t \delta x_t  }_{\delta l_t}
\end{align}

\begin{theorem}
\begin{align}
\Tilde{\mathcal{J}}  &= \sigma^{-1} \ln( \alpha ) +  \bar{v}_0 -  \frac{1}{2} v_0^T ( V_{0} + \sigma  ^{-1}\chi^{-1}_{0})^{-1} v_0 \label{eqn:line-search-cost-approx}
\end{align}
with
\begin{align}
    \alpha &= \vert I + \sigma V_0 \chi_0 \vert ^ \frac{1}{2} \ \Pi_{t=1}^T \vert I + \sigma V_t \Omega_t \vert ^ \frac{1}{2}
\end{align}
and where $ \Bar{v}_0, v_0, V_0$ are the terminal values of the recursion:
\begin{subequations}\label{eqn:cost-approx-value}
\begin{align}
    V_t =& \ l_{t}^{xx} + {f^x_t}^T M_t f^x_t \\
    v_t =& \ l_{t}^{x} + {f^x_t}^T M_t \bar{f}_{t+1} +  {f^x_t}^T N_t \\
    \bar{v}_t =& \  \bar{l}_{t}  + \Bar{v}_{t+1}  + \bar{f}_{t+1}^T N_t + \frac{1}{2} \bar{f}_{t+1}^T M_t \bar{f}_{t+1}  \\
    & - \frac{1}{2}  v_{t+1}^T (V_{t+1} +  \sigma^{-1} \Omega_{t+1}^{-1})^{-1}  v_{t+1}    \nonumber
\end{align}
\end{subequations}
where $M_t$ and $N_t$ have the same form as~\eqref{eqn:Mt} and~\eqref{eqn:Nt} but evaluated at $V_t$ and $v_t$ given in~\eqref{eqn:cost-approx-value} 

\end{theorem}
 \begin{proof}
 The proof follows a similar logic of theorem~\ref{thm:optimal-controls} and is provided fully in the repository. 
 \end{proof}

\subsection{Step Length \& Partial Update of the Gaps} 

We choose to update the gaps partially as is done in~\cite{mastalli2020crocoddyl}. This allows for a partial contraction of the gaps unless a full step $\alpha=1$ is taken thus allowing for a better convergence when initialized from infeasible nominal trajectories, the update equations are then

\begin{align}
    u^{i+1}_t  &= u^i_t -  \alpha k_t - K_t\left(x^{i+1}_t - x^{i}_t\right),  \label{eqn:forward-partial-control-update}\\
    x^{i+1}_{t+1}& = f_t\left(x^{i+1}_t ,  u^{i+1}_t  \right) - \left(1 - \alpha \right) \bar{f}^{i}_{t+1} . \label{eqn:forward-partial-gaps}
\end{align}

with the initial condition $x^{i+1}_0 =  x^{i}_0$. We perform a simple backtracking line search for a set of values of $\alpha$ until no more decrease in the approximated cost is observed. The pseudo-code algorithm is provided in Algorithm~\ref{alg:planning}.

\begin{algorithm}
\DontPrintSemicolon
  \KwInput{nonlinear cost, nonlinear dynamics, nonlinear measurement, sensitivity, $\mathcal{X}^0$, $\mathcal{U}^0$}
  \textbf{Initialization:} compute dynamics \& cost approximations~\eqref{eqn:dynamics_approx},~\eqref{eqn:cost_approx} \; 
  \While{Not Converged}
  {
    \tcp{backward pass}
  	\For{$t=T$; $t\geq 0$; $t$ - - }{
  	$k_t, K_t \gets eqn.~\eqref{eqn:optimal-control-deviation}$ \; 
  	$ V_t, v_t, \bar{v}_t \gets eqn.~\eqref{eqn:value_recursion}$ \; 
  	}
  	\tcp{forward pass} 
  	\For{$\alpha \in \mathcal{A} $}{
  	    \For{$t=0$; $t++$; $t\leq T$}{
  		$u^{i+1}_t \gets eqn.~\eqref{eqn:forward-partial-control-update}$  \;
  		$x^{i+1}_t \gets eqn.~\eqref{eqn:forward-partial-gaps}$ \;
  		}
  		$\mathcal{J}^{i+1} \gets  eqn.~\eqref{eqn:line-search-cost-approx}$ \;   
  		\If{$\mathcal{J}^{i+1} \leq \mathcal{J}^{i}$ }
  		    {
  		    $\mathcal{X}, \mathcal{U} \gets x^{i+1}_{0:T}, u^{i+1}_{0:T-1}$ \;
  		    compute dynamics \& cost approximations~\eqref{eqn:dynamics_approx},~\eqref{eqn:cost_approx} \; 
  		    \textbf{break} \;}
  	}
  	
  	\If{$\vert \mathcal{J}^{i+1} - \mathcal{J}^{i} \vert \leq 10^{-12}$  }
  	{
  	    $Converged \gets True$ \; 
  	}
  }
\KwOutput{$\mathcal{X}$, $\mathcal{U}$, $K_{0:T-1}$, $V_{0:T}$, $v_{0:T}$}
\caption{Planning Phase}\label{alg:planning}
\end{algorithm}

\section{Estimation \& Optimal Controls}\label{sec:estimation}

The actual optimal control problem of the imperfectly observable system is not fully solved yet. So far we only have the optimal control if optimal $x_t$ were known. The next step consists of solving for the past stress 

\subsection{Past Stress Optimization}

The solution to the past stress~\eqref{past_stress_reccursion} optimization with the cost approximation~\eqref{eqn:line-search-cost-approx}  and linear dynamics~\eqref{eqn:dynamics_approx} can be obtained using a forward recursion as described in the following theorem:

\begin{theorem}\label{thm:past-stress-recursion}
\begin{align}
  \sigma  \mathcal{P}( \delta x_t, W_t) = \frac{1}{2} \left(\delta x_t - \delta \hat{x}_t\right)^T P^{-1}_t \left( \delta x_t - \delta  \hat{x}_t\right)  
\end{align}
where $P_t$ and $\delta \hat{x}_t$ are subject to the following recursion
\begin{align}\label{eqn:estimator-equation}
\delta \hat{x}_{t+1} =& \ f^{x}_{t} \delta \hat{x}_{t} + f^{u}_{t} \delta u_{t} + \bar{f}_{t+1} + G_t \left( \delta y_{t+1} - h^{x}_{t} \delta \hat{x}_{t}\right)  \nonumber  \\
&-   \sigma  f^{x}_t H_t \left(    l^{xx}_t  \delta \hat x_t +  l^{xu}_{t} \delta u_{t} +   l^x_{t}  \right) 
\end{align}
\begin{subequations}\label{eqn:estimator-matrices}
\begin{align}
    H_t &= \left( P^{-1}_{t}  + h^{x^T}_{t} \Gamma^{-1}_{t+1}  h^{x}_{t} + \sigma l^{xx}_t    \right) ^{-1}\\
    G_t &= f^{x}_{t} H_t  h^{x^T}_{t} \Gamma^{-1}_{t+1} \label{eqn:estimator_gain}\\ 
    P_{t+1} &= \Omega_{t+1} + f^{x}_{t}  H_t f^{x^T}_{t}\label{eqn:estimator_curvature}
\end{align} 
\end{subequations}
\end{theorem} 
\vspace{0.4cm}

\begin{proof}
The proof is done by induction. Given $\mathcal{P}\left(\delta x_{t-1}, W_{t-1}\right)$, We first notice that:
\begin{align}
     \sigma l_{t-1} &+  \sigma \mathcal{P}\left(\delta x_{t-1}, W_{t-1}\right) = \\
     &\frac{1}{2} \left( \delta x_{t-1} - \delta \Tilde{x}_{t-1} \right)^T \Tilde{P}^{-1}_{t-1} \left(\delta x_{t-1} - \delta \Tilde{x}_{t-1} \right) + C^{ste} \nonumber
\end{align}
where:
\begin{subequations}
\begin{align}
    \Tilde{P}_{t-1} &= \left(P^{-1}_{t-1} + \sigma l^{xx}_{t-1} \right)^{-1} \\ 
    \delta \Tilde{x}_{t-1} &= \Tilde{P}_{t-1} \left(  P^{-1}_{t-1} \delta \hat{x}_{t-1}  -   \sigma l^{xu}_{t-1} \delta u_{t-1}  -  \sigma l^x_{t-1}  \right)
\end{align}
\end{subequations}
Hence,~\eqref{past_stress_reccursion} can be written as:

\begin{align}
    \sigma \mathcal{P} & \left(\delta x_{t}, W_t \right) =  C^{ste}\\
    & + \underset{\delta x_{t-1}}{\ext} \,\, d_{t} + \frac{1}{2}  \left(\delta x_{t-1} - \delta \Tilde{x}_{t-1} \right)^T \Tilde{P}^{-1}_{t-1} \left( \delta x_{t-1} - \delta \Tilde{x}_{t-1} \right)  \nonumber
\end{align}

Carrying out this optimization over $\delta x_{t-1}$ and rearranging the terms to construct a perfect square concludes the proof. This last step is included in the accompanying repository. 
\end{proof}

\subsection{Minimum Stress State \& Optimal Control}

Now that both $\mathcal{P}\left(\delta x_t, W_t\right)$ and $\mathcal{F}(\delta x_t)$ are obtained, what remains is optimizing the last remaining variable $\delta x_t$, i.e. solving 
\begin{align}
   \delta \check{x}_t &= \underset{\delta x_t}{\operatorname{argmin}} \quad \mathcal{P}(\delta x_t, W_t) + \mathcal{F}(\delta x_t) \nonumber \\
     &= \left(I + \sigma P_t V_t\right)^{-1}\left(\delta \hat{x}_t - \sigma P_t v_t\right)\label{eqn:minimal_stress_estimate}
\end{align}
which results in the minimum stress state deviation $\delta \check{x}_t$. And finally we obtain the optimal control law 

\begin{align} 
   u^{*}_{t} &=  u^n_t -  K_t \left(I + \sigma P_t V_t\right)^{-1} \left(\delta \hat{x}_t - \sigma P_t v_t\right)  \label{eqn:final-optimal-control}
\end{align}

\begin{algorithm}
\DontPrintSemicolon
  \KwInput{$\mathcal{X}$, $\mathcal{U}$, $K_{0:T-1}$, $V_{0:T}$, $v_{0:T}$, $\delta y_{t}$, $\delta \hat{x}_{t-1}$, $P_{t-1}$}
  
  $H_{t-1}, G_{t-1}, P_{t} \gets eqn.~\eqref{eqn:estimator-matrices}$ \; 
  $\delta \hat{x}_{t} \gets eqn.~\eqref{eqn:estimator-equation}$ \; 
  $\delta \check{x}_t \gets eqn.~\eqref{eqn:minimal_stress_estimate}$ \; 
  $u^\star_t \gets eqn.\eqref{eqn:final-optimal-control}$ \;
\KwOutput{$ u^{\star}_t$}
\caption{Run Time Estimation \& Control}\label{alg:estimation}
\end{algorithm}

\section{Results}\label{sec:results}




To demonstrate the capabilities of the algorithm with nonlinear systems with measurement noise, we optimize the motion of a one dimensional pneumatic hopper depicted in~\ref{fig:snapshots}. The state vector is four dimensional $x^T=[q^h,q^f, v^h,v^f]$ including the hip position relative to the ground, the foot position relative to the hip position, and their velocities respectively. The control vector is one dimensional controlling the acceleration of the foot relative to the hip. The distance from the foot position to the ground can then be defined as $e=q^f - q^h$. 

\begin{figure}[ht]
   \centering
   \includegraphics[width=0.5\linewidth]{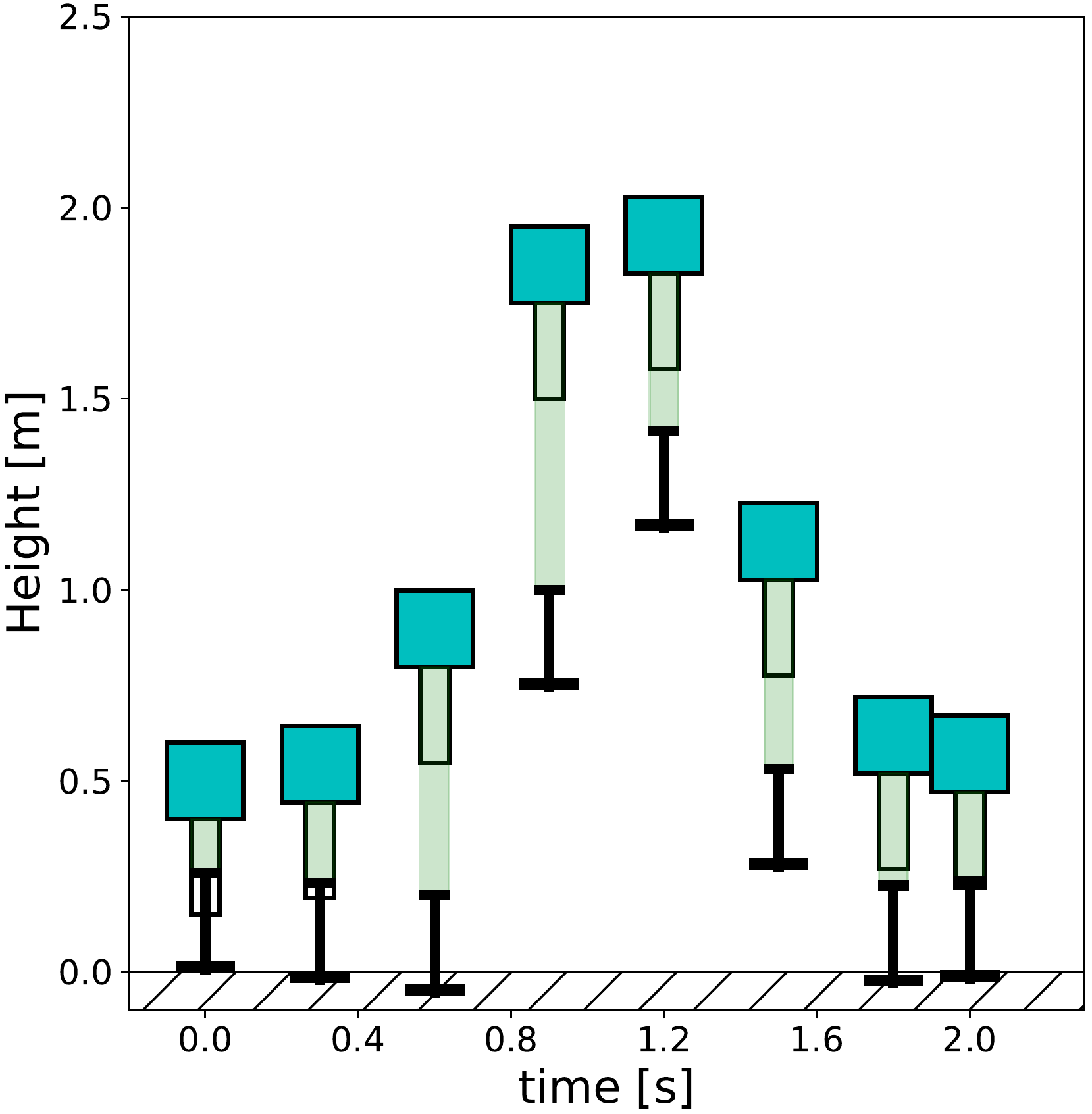}  
   \caption{Simulation snapshots with pneumatic hopper.}
   \label{fig:snapshots}
\end{figure}

The simulation uses a stiff visco-elastic contact model $\lambda(e, \Dot{e}) = k e - b \Dot{e}$. For optimization we will use the relaxed contact model 

\begin{equation}
  \lambda(e)=\begin{cases}
    0., & \text{if $e<0$}.\\
    \frac{k}{2\alpha} e^2, & \text{if $0 \leq e < \alpha$}. \\ 
    k e - \frac{1}{2} k \alpha & \text{otherwise}. 
  \end{cases}
\end{equation}

To solve the optimal control problem, we use $\alpha = 0.01$ and $k=500$, where the simulation parameters use $k=10^{5}$ and $b=300$ (we aim to demonstrate that feedback can compensate for the unexpectedly stiff contact in simulation). The cost function used includes three phases, the stance phase, which is along the horizon $T$ except for $t = T/2$, the jumping cost for $t = T/2$ and terminal cost for time $t=T$. We choose to write the original cost as $l_t = \frac{1}{2}(x_t-x_{des})^T l^{xx}_t(x_t-x_{des}) + \frac{1}{2} u^T_t l^{uu}_t u_t$ where $x^T_{des} = [0.5, 0., 0., 0.]$ for the stance and terminal costs and  $x^T_{des}= [2., 0., 0., 0.]$ for the jump phase. The cost weights are set as follows, for the stance phase $ l^{xx} = diag(10^1, 1, 10^{-4}, 0.)$, for the jump phase $ l^{xx} = diag(10^1, 10^{-2}, 10^{-1}, 0.)$, for the terminal phase we have $ l^{xx} = diag(10^1, 10^1, 1., 1.)$, and  $ l^{uu} = 10^{-3}$ for all running costs. The noise is assumed to be uniform with $\Omega_t = \Gamma_t = diag(10^{-3}, 10^{-3}, 10^{-3}, 10^{-3})$.  One thousand simulations are carried out for each of the three scenarios, a risk neutral controller and estimator namely DDP~\cite{mastalli2020crocoddyl} combined with an extended kalman filter. A risk averse controller with $\sigma=-0.5$ and a risk seeking controller with $\sigma = 10$. All the results below are presented in terms of the mean/average and standard deviation of the data obtained. Risk Sensitive behavior achieves superior performance compared to the two other formulations as we will show below.  

\begin{figure}[ht]
\centering
\captionsetup{justification=centering}
\subfloat[][DDP Plan]{
\includegraphics[width=0.5\linewidth]{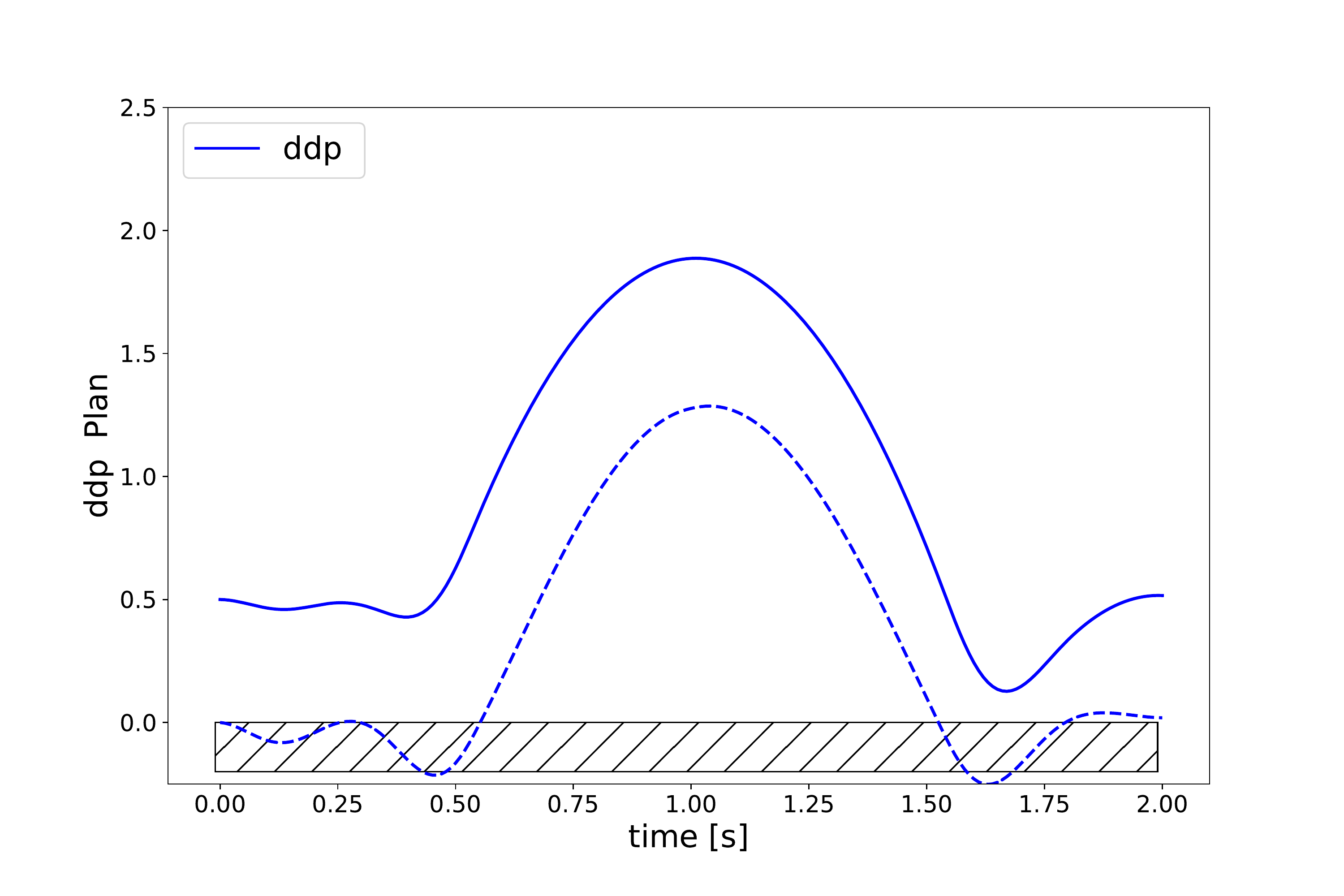}
\label{fig:ddp-plan}} 
\subfloat[][Risk Averse Plan]{
\includegraphics[width=0.5\linewidth]{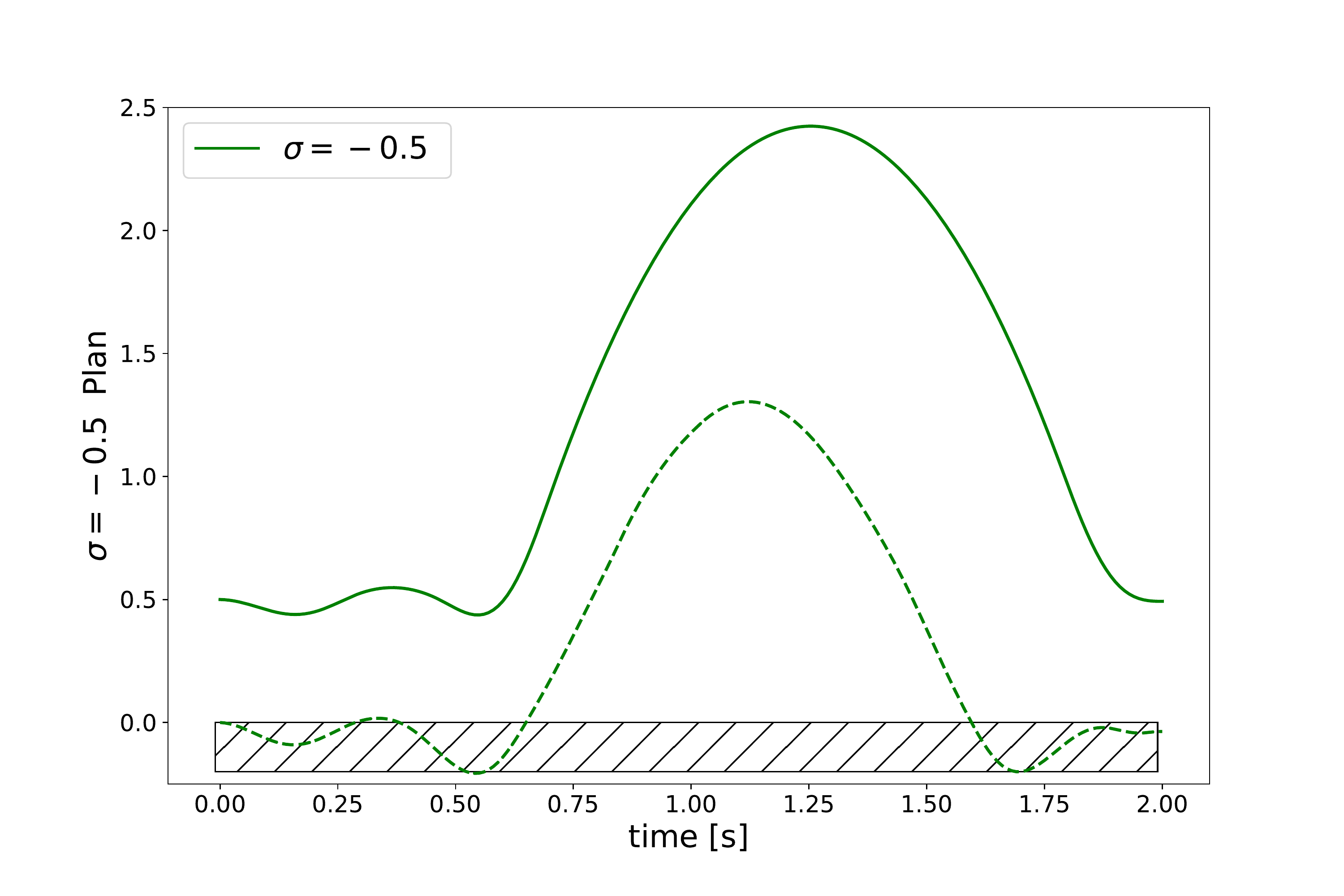}
\label{fig:averse-plan}} \\
\subfloat[][Risk Seeking Plan]{
\includegraphics[width=0.5\linewidth]{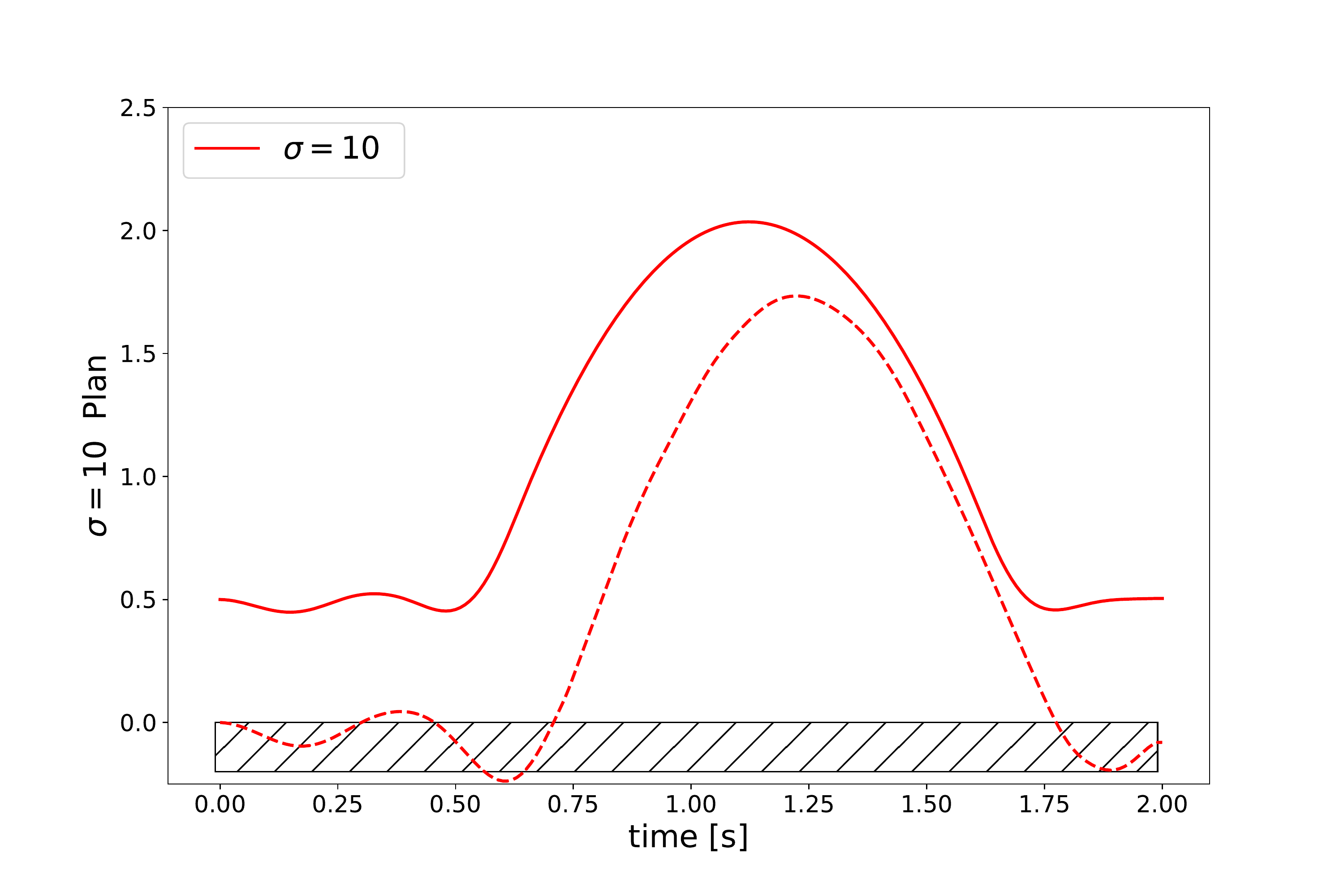}
\label{fig:seeking-plan}} 
\caption{Planned Trajectories, Solid lines represent the hip trajectory and dashed lines represent the foot trajectory.}
\label{fig:planned-trajectories}
\vspace{-0.3cm}
\end{figure}

\subsection{Optimized Trajectories}

One advantage of the proposed procedure is that it is capable of generating different trajectories depending on the noise and sensitivity parameters, this is evident from the resulting plans. As the results in Fig.~\ref{fig:planned-trajectories} show, both DDP and risk seeking scenarios generate a plan that reaches exactly the desired height Fig.~\ref{fig:ddp-plan} and Fig.~\ref{fig:seeking-plan}, whereas the risk averse scenario generates a higher maximum height Fig.~\ref{fig:averse-plan}. The penetration into the ground in the plan is due to using a relaxed contact model that does not generate stiff contact behaviors. The goal is to have the control feedback compensate for the stiff behavior in the simulations.

\subsection{Simulation Statistics}

In order to compare performance, the statistics of the accumulated original cost $\mathcal{L}$ is computed for each of the three scenarios studied and summarized in Table~\ref{tab:cost-stats}. Risk averse design achieves a lower cost average, and a significantly tighter standard deviation. Whereas risk seeking design generates the highest cost average and ddp generates the highest standard deviation (Fig.~\ref{fig:cost-stats}).

\begin{table}[ht]
\centering
\captionsetup{justification=centering}
\begin{tabular}{|c|c|c|c|} \hline
    & DDP &  $\sigma= -0.5$ & $\sigma= 10$\\ \hline
\rowcolor[gray]{.9} Total average &  $5.79 \times 10^{-2}$ & $5.13 \times 10^{-2}$ & $13.92 \times 10^{-2}$  \\ 
Maximum Std & $13.17 \times 10^{-2}$ & $3.53 \times 10^{-2}$ & $12.92 \times 10^{-2}$ \\ \hline
\end{tabular}
\caption{Cost Statistics}
\label{tab:cost-stats}
\vspace{-0.4cm}
\end{table}

\begin{figure}[ht]
   \centering
    \captionsetup{justification=centering}
   \includegraphics[width=0.5\linewidth]{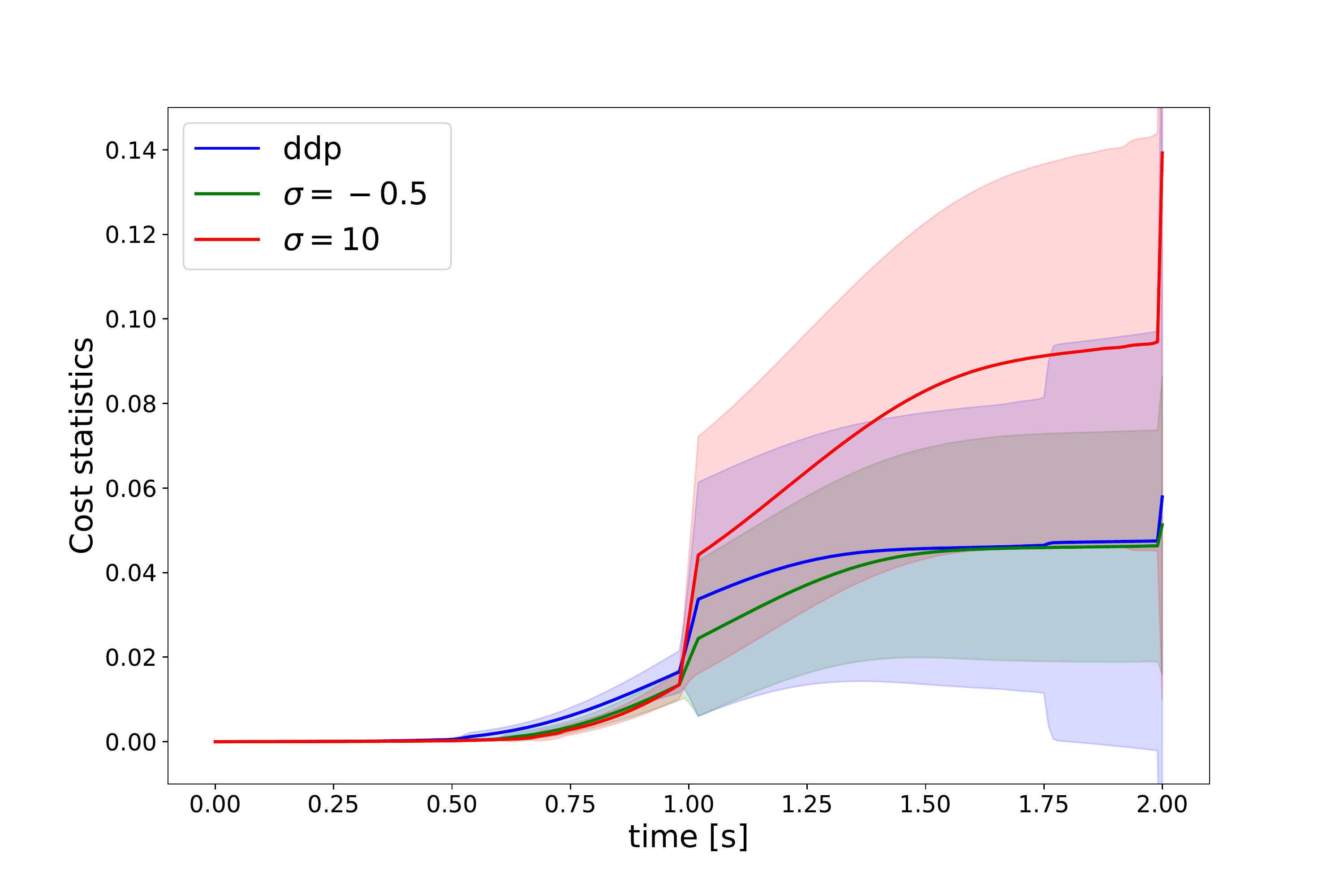}  
   \caption{Cost Statistics for DDP, Risk Averse and Risk Seeking control designs after 1000 simulations.}
   \label{fig:cost-stats}
\end{figure}

The tracking error statistics is computed with respect to the optimized plans for each of the scenarios, and presented for each of the states in Table~\ref{tab:tracking-statistics}. We see that the risk-averse design maintains the lowest average position errors and lowest standard deviations for all four states. Although the planned trajectory aims for a higher height, the actual average hip height reaches the desired hip height as shown in Fig.~\ref{fig:state-statistics}.
\begin{table}[ht]
\centering
\captionsetup{justification=centering}
\begin{tabular}{|c|c|c|c|c|c|c|} \hline
\multirow{2}{*}{State} & \multicolumn{3}{c|}{Max. Average} & \multicolumn{3}{c|}{Max. STD} \\ \cline{2-7}
   &DDP & $\sigma= -0.5$ & $\sigma= 10$ & DDP & $\sigma= -0.5$ & $\sigma= 10$ \\  \hline
\rowcolor[gray]{.9} $q^h$ & 0.354 & 0.334 & 0.410 & 0.248 & 0.230 & 0.244  \\ 
                    $q^f$ & 0.351 & 0.306 & 0.370 & 0.133 & 0.122 & 0.150 \\
\rowcolor[gray]{.9} $v^h$ & 3.83 & 3.337 & 3.135  & 1.815 & 1.652 & 1.585 \\ 
                    $v^f$ & 2.35 & 2.182 & 3.959  & 1.256 & 0.743 & 1.704 \\ \hline
\end{tabular}
\caption{Tracking Error Statistics}
\label{tab:tracking-statistics}
\vspace{-0.4cm}
\end{table}

\begin{figure}[ht]
\centering
\captionsetup{justification=centering}
\subfloat[][Hip position]{
\includegraphics[width=0.5\linewidth]{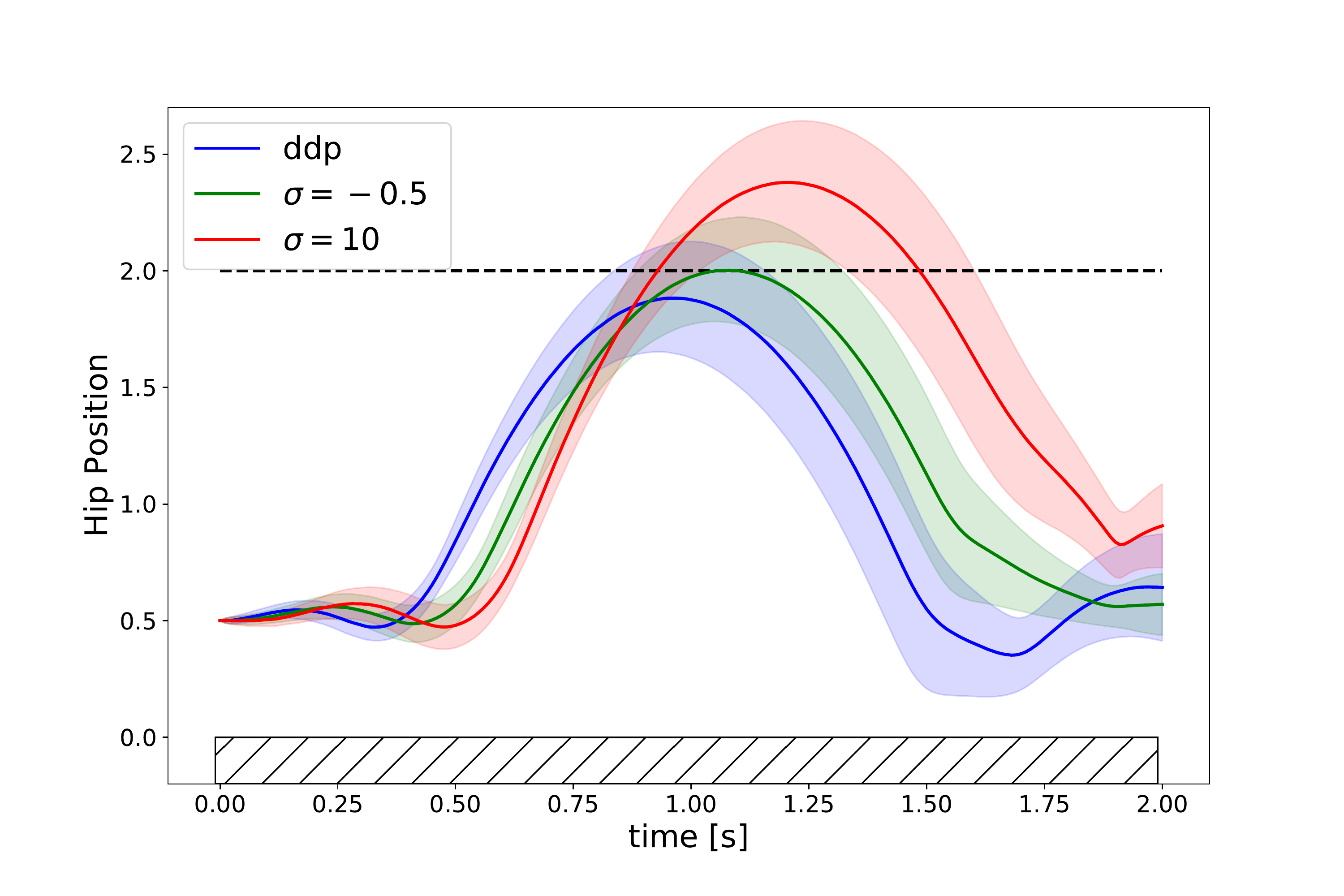}
\label{fig:hip-pos-stat}} 
\subfloat[][Foot position]{
\includegraphics[width=0.5\linewidth]{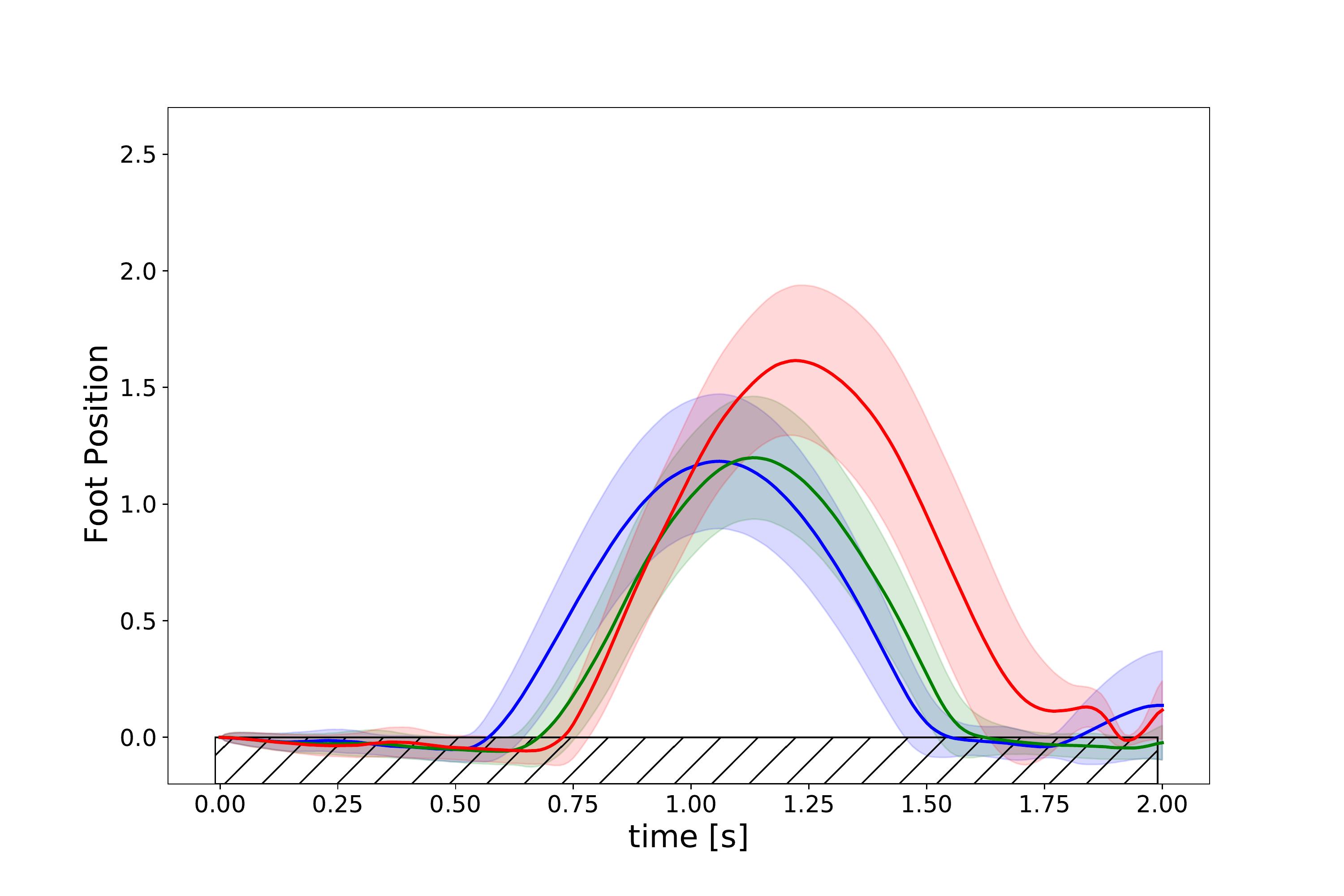}
\label{fig:foot-pos-stat}}
\caption{Hip and Foot Simulated Position Statistics}
\label{fig:state-statistics}
\vspace{-0.3cm}
\end{figure}

On the other hand, DDP comes short from the desired height, and experiences a foot bounce off the ground more often as shown in Fig.~\ref{fig:foot-pos-stat}. The risk seeking design achieves a higher height than desired, much later than planned and also exhibits a foot bouncing back off the ground (Fig. ~\ref{fig:state-statistics}). 


Another interesting aspect is the maximum average forces generated. Even though the forces are not explicit states that are penalized in the cost, the force plays a crucial role in taking off, reaching the desired height and landing safely. DDP generates the lowest takeoff force, which explains the lower average peak height (Table~\ref{tab:maximum-average}). For the landing phase DDP generates the highest impact force (Fig.~\ref{fig:average-force}) while risk seeking generates the lowest, but then diverges. Risk averse generates the most stable contact force profiles, a highly sought behavior for systems establishing and breaking contact with the environment instead of avoiding it.

\begin{table}[ht]
\centering
\captionsetup{justification=centering}
\begin{tabular}{|c|c|c|c|c|} \hline
     & Takeoff $f$ & Landing $f$ & Hip Height & Foot Height\\ \hline
\rowcolor[gray]{.9}            DDP  & 118.099 & 106.902 & 1.888 & 0.698 \\ 
                    $\sigma = -0.5$ & 106.146 & 100.527 & 2.007 & 0.801 \\ 
\rowcolor[gray]{.9} $\sigma = 10.$  & 103.712 & 330.583 & 2.380 & 0.816 \\ \hline
\end{tabular}
\caption{Maximum average force at takeoff and landing, hip height and foot height}
\label{tab:maximum-average}
\vspace{-0.4cm}
\end{table}

\begin{figure}[!tbp]
   \centering
   \includegraphics[width=0.5\linewidth]{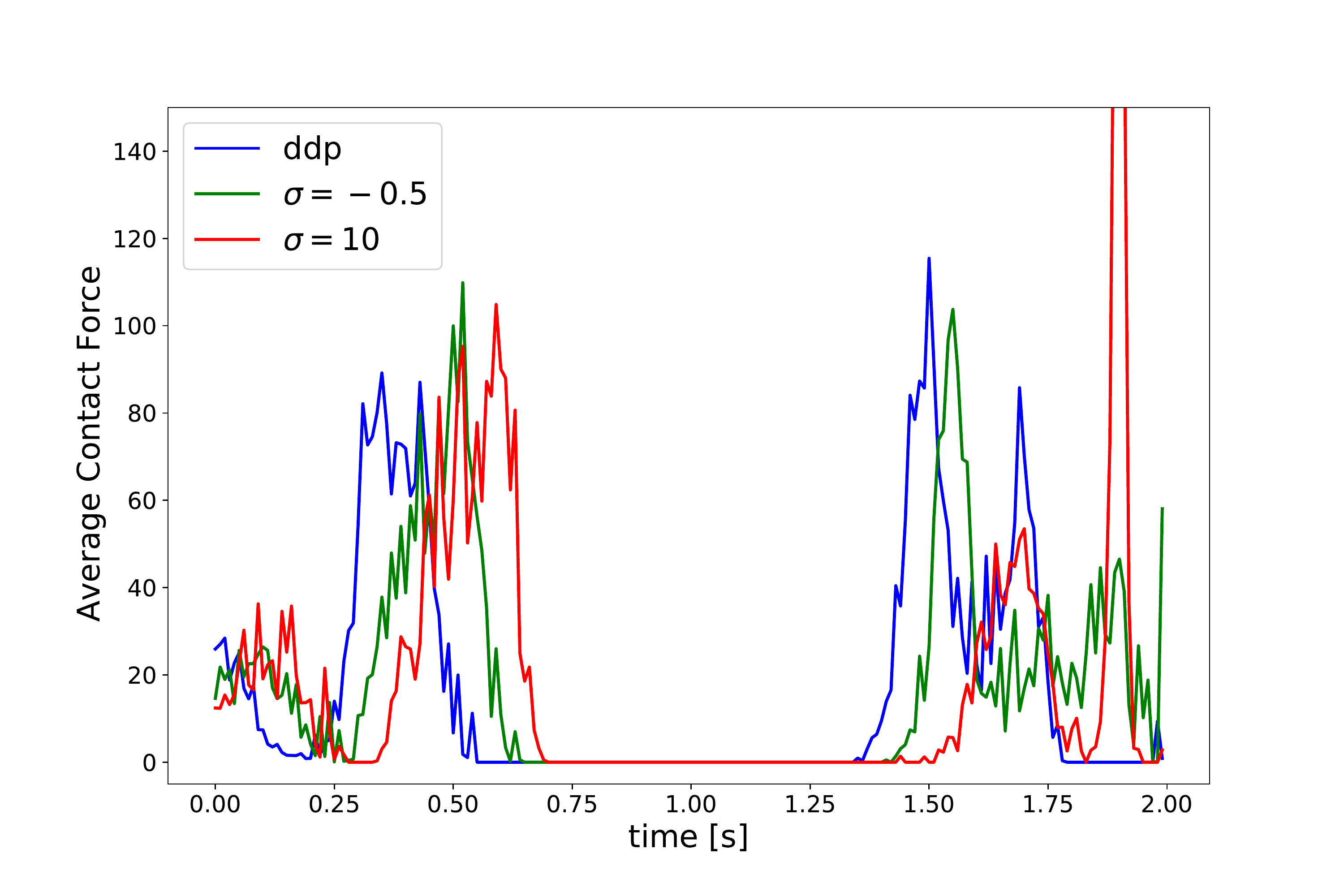} 
   \caption{Average contact force}
   \label{fig:average-force}
\end{figure}

\section{Conclusion}
In this paper we present, to our knowledge, the first computationally efficient iterative algorithm to compute risk-sensitive optimal trajectories and control policies that handle both process and measurement noise. The plan is obtained before any observations are collected, as observations become available, an estimator that takes into account the cost accumulated generates state estimates that are then used to update both the control signal and the feedback controls. We tested the proposed algorithm on a one dimensional hopper in intermittent contact with the environment, subject to both model disturbances and measurement noise. In the risk averse scenario the proposed algorithm proves to be superior to its deterministic counter part namely DDP.


\bibliographystyle{IEEEtran}
\bibliography{references}

\appendix

\subsection{Integration of Exponential of a Quadratic form.}\label{lemma:integration_eofq}

\begin{lemma} 
For a quadratic form in $x$ and $y$
\begin{align}
    Q(x,y) &=  \frac{1}{2} \begin{bmatrix} 1 \\ x \\ y\end{bmatrix}^T \begin{bmatrix} \bar{q} & q_x^T & q_y^T \\ q_x & Q_{xx} & Q_{xy} \\ q_y & Q_{yx} & Q_{yy} \end{bmatrix} \begin{bmatrix} 1\\  x\\ y\end{bmatrix}
\end{align}
We have: 
\begin{align}
     \int \exp \left(-Q(x,y)\right) dx &= \vert 2\pi Q^{-1}_{xx}\vert ^{\frac{1}{2}} \exp\left(- Q(\hat{x}, y) \right) \\
     \hat{x} = \underset{x}{\operatorname{argmin}}\ \ Q(x, y) &= -Q^{-1}_{xx} \left(Q_{xy}y + q_x\right)
\end{align}

\end{lemma}
\hspace{0.1cm}
\begin{proof}
We notice that:
\begin{align}
    Q(x, y) - Q(\hat x, y) &= \frac{1}{2} (x - \hat x)^T Q_{xx} (x - \hat x) 
\end{align}
Then, the following equality concludes the proof:
\begin{equation}
\int \exp\left(- \frac{1}{2} ( x - \hat{x} )^T Q_{xx}( x - \hat{x} )   \right) dx = \sqrt{\vert 2\pi Q^{-1}_{xx}\vert }
\end{equation}
\end{proof}

\subsection{main theorem}\label{app:principle_of_optimality}

\begin{theorem}

\begin{align}
    \mathbb{E}_{\pi^{\star}} \left[ \exp\left(- \sigma \mathcal{L}\right) \ \vert \ W_t   \right] &=  \frac{\alpha_t}{p(W_t)} \exp\left( - \sigma \mathcal{S}_t(W_t) \right) \label{Whittle_theorem}
\end{align}

\end{theorem}
\begin{proof}

According to~\cite{whittle1982optimization}, we have

\begin{align}
    & \mathbb{E}_{\pi^{\star}}  \left[ \exp\left(-  \sigma \mathcal{L}\right)  \vert W_t \right] \label{optimality_eq} \\ \nonumber
    &= \ext_{u_t} \int \mathbb{E}_{\pi^{\star}} \left[ \exp\left(-  \sigma \mathcal{L}\right) \vert W_{t+1} \right]  p(y_{t+1} | W_t) dy_{t+1} 
\end{align}

The proof is then a backward induction based on the recursion equality \eqref{optimality_eq}. We first notice that:  
\begin{align}
     \mathbb{E}_{\pi^{\star}} & \left[ \exp\left(-  \sigma \mathcal{L}\right)  \vert W_T \right]\\
     &=  \int e^{- \sigma \mathcal L} \, p(x_{0:T} | W_T) dx_{0:T} \nonumber \\
    &= \frac{1}{p(W_T)} \int e^{- \sigma \mathcal L} \, p(x_{0:T}, y_{0:T}) dx_{0:T} \nonumber \\
& = \alpha_T \exp ( -\sigma \min_{x_0, ...x_T} \mathcal S) \nonumber 
\end{align}

The last equality comes from the lemma therefore, we have:
\begin{equation}
    \alpha_T = \frac{ \vert 2\pi \sigma \frac{\partial^2 \mathcal S}{\partial^2 x_{0:T}}\vert^{\frac{1}{2}} }{p(W_T) \vert 2 \pi \chi_0 \vert ^{\frac{1}{2}} \prod_{t=1}^T \vert 2 \pi \Omega_t \vert ^{\frac{1}{2}}} 
\end{equation}

Now, assuming the Equation \ref{Whittle_theorem} for $t+1$, let's show the property at time $t$:
Thanks to \eqref{optimality_eq}, we have

\begin{align}
    \mathbb{E}_{\pi^{\star}} & \left[ \exp\left(-  \sigma \mathcal{L}\right)  \vert W_t \right] \\
    =& \ext_{u_t} \int  \frac{\alpha_{t+1}}{p(W_{t+1})} \operatorname{exp} \left[ -  \sigma ( \mathcal{S}_{t+1}(W_{t+1}) ) \right] p(y_{t+1} | W_t) dy_{t+1} \nonumber \\
    & =   \frac{\alpha_{t+1}}{p(W_{t})} \ext_{u_t}  \int \operatorname{exp} \left[ -   \sigma  \mathcal{S}_{t+1}(W_{t+1})  \right] dy_{t+1} \nonumber \\
    & =  \frac{\alpha_{t}}{p(W_{t})} \ext_{u_t}  \ext_{y_{t+1}}  \operatorname{exp} \left[ -  \sigma  \mathcal{S}_{t}(W_{t})  \right] \nonumber 
\end{align}

The last equality comes from the lemma therefore, we have:

\begin{align}
 \alpha_t = \alpha_{t+1} \bigg| 2\pi \sigma \frac{\partial^2 \mathcal \mathcal{S}_{t+1}(W_{t+1}) }{\partial^2 y_{t+1}} \bigg| ^{\frac{1}{2}}
\end{align}

\end{proof}










\end{document}